\def\CTeXPreproc{Created by ctex v0.2.12, don't edit!}%Corrected on Dec. 08, 2007; by samy
\documentclass[a4paper,12pt]{amsart}
\usepackage{amssymb}
\usepackage{ifthen}
\usepackage{graphicx}
\nonstopmode
\newtheorem{thm}{Theorem}
\newtheorem{cor}{Corollary}
\newtheorem{lem}{Lemma}
\newtheorem{prop}{Proposition}

\newtheorem{claim}{Claim}

\newtheorem{conj}{Conjecture}
\newtheorem{prob}{Problem}

\theoremstyle{definition}
\newtheorem{defn}{Definition}
\newtheorem{example}{Example}

\newenvironment{rem}{%
\bigskip
\noindent \textsl{{\sl Remark. }}}{\bigskip}
\newenvironment{rems}{%
\bigskip
\noindent \textsl{{\sl Remarks. }}}{\bigskip}

\newenvironment{pf}[1][]{%
 \vskip 1mm
 \noindent
 \ifthenelse{\equal{#1}{}}%
  {{\slshape Proof. }}%
  {{\slshape #1.} }%
 }%
{\qed\medskip}

\newcounter{alphabet}
\newcounter{tmp}
\newenvironment{Thm}[1][]{\refstepcounter{alphabet}%
\bigskip%
\noindent%
{\bf Theorem \Alph{alphabet}}%
\ifthenelse{\equal{#1}{}}{}{ (#1)}%
{\bf .} \itshape}{\vskip 8pt}
\makeatletter
\newcommand{\Ref}[1]{\@ifundefined{r@#1}{}{\setcounter{tmp}{\ref{#1}}\Alph{tmp}}}
\makeatother
\newcommand{\IR}{{\mathbb R}}
\newcommand{\IN}{{\mathbb N}}

\def\be{\begin{equation}}
\def\ee{\end{equation}}

\newcommand{\bee}{\begin{enumerate}}
\newcommand{\eee}{\end{enumerate}}

\newcommand{\blem}{\begin{lem}}
\newcommand{\elem}{\end{lem}}
\newcommand{\bthm}{\begin{thm}}
\newcommand{\ethm}{\end{thm}}
\newcommand{\bcor}{\begin{cor}}
\newcommand{\ecor}{\end{cor}}
\newcommand{\beg}{\begin{example}}
\newcommand{\eeg}{\end{example}}
\newcommand{\begs}{\begin{examples}}
\newcommand{\eegs}{\end{examples}}
\newcommand{\bdefe}{\begin{defn}}
\newcommand{\edefe}{\end{defn}}
\newcommand{\bprob}{\begin{prob}}
\newcommand{\eprob}{\end{prob}}
\newcommand{\bques}{\begin{ques}}
\newcommand{\eques}{\end{ques}}
\newcommand{\bei}{\begin{itemize}}
\newcommand{\eei}{\end{itemize}}

\newcommand{\bde}{\begin{deter}}
\newcommand{\ede}{\end{deter}}
\newcommand{\bca}{\begin{case}}
\newcommand{\eca}{\end{case}}
\newcommand{\bcl}{\begin{claim}}
\newcommand{\ecl}{\end{claim}}
\newcommand{\bcon}{\begin{conj}}
\newcommand{\econ}{\end{conj}}
\newcommand{\bcons}{\begin{conjs}}
\newcommand{\econs}{\end{conjs}}
\newcommand{\bprop}{\begin{propo}}
\newcommand{\eprop}{\end{propo}}
\newcommand{\br}{\begin{rem}}
\newcommand{\er}{\end{rem}}
\newcommand{\brs}{\begin{rems}}
\newcommand{\ers}{\end{rems}}
\newcommand{\bo}{\begin{obser}}
\newcommand{\eo}{\end{obser}}
\newcommand{\bos}{\begin{obsers}}
\newcommand{\eos}{\end{obsers}}
\newcommand{\bpf}{\begin{pf}}
\newcommand{\epf}{\end{pf}}
\newcommand{\ba}{\begin{array}}
\newcommand{\ea}{\end{array}}
\newcommand{\beq}{\begin{eqnarray}}
\newcommand{\beqq}{\begin{eqnarray*}}
\newcommand{\eeq}{\end{eqnarray}}
\newcommand{\eeqq}{\end{eqnarray*}}

\newcommand{\ds}{\displaystyle}

\newcounter{minutes}
\divide\time by 60
\newcounter{hours}
\multiply\time by 60 \addtocounter{minutes}{-\time}
\begin{document}
\title[Compositions of polyharmonic mappings]{Compositions of polyharmonic mappings}

%=========================================================================
\thanks{%$^\dagger$
File:~\jobname .tex,
          printed: \number\day-\number\month-\number\year,
          \thehours.\ifnum\theminutes<10{0}\fi\theminutes}
%=========================================================================

\author[G. Liu]{Gang Liu}
\address{Gang Liu, College of Mathematics and Statistics (Hunan Provincial Key Laboratory of Intelligent Information Processing and Application),
Hengyang Normal University, Hengyang,  Hunan 421008, People's Republic of China.}
\email{liugangmath@sina.cn}

\author[S. Ponnusamy]{Saminathan Ponnusamy
%$^\dagger $
%${}^{~\mathbf{*}}$
}
\address{Saminathan  Ponnusamy,
Indian Statistical Institute (ISI), Chennai Centre, SETS (Society
for Electronic Transactions and Security), MGR Knowledge City, CIT
Campus, Taramani, Chennai 600 113, India.}
\email{samy@isichennai.res.in, samy@iitm.ac.in}

%; Secondary: 30C20, 30C55}
%,  31B05, 31C05}

\keywords{analytic, anti-analytic and polyharmonic.
%\\
%$%{}^{\mathbf{*}}
%^\dagger$ {\tt This author is on leave from the Department of Mathematics,
%Indian Institute of Technology Madras, Chennai-600 036, India}
}
%\thanks{ }
%\maketitle

%\newfont{\Bbb}{msbm10 scaled\magstephalf}
\begin{abstract}
The paper is devoted to the study of compositions of polyharmonic mappings in simply connected domains.
More precisely, we determine necessary and sufficient conditions of polyharmonic mapping $f$ such that
$f\circ F$ (resp. $F\circ f$) is $l$-harmonic for any analytic function (or harmonic mapping but not analytic, or $q$-harmonic
mapping but not $(q-1)$-harmonic) $F$.
%, where $q$ and $l$ are positive integers.
\end{abstract}

\maketitle \pagestyle{myheadings}
\markboth{G. Liu and S. Ponnusamy}{Compositions of Polyharmonic Mappings}

\section{Introduction and Main Results}
A complex-valued function $f$ is called a harmonic mapping if it satisfies the harmonic equation $\Delta f=0$,
where $\Delta$ denotes  the complex Laplacian operator
$$\Delta=4\frac{\partial^2}{\partial z\partial\overline{z}}=\frac{\partial^2}{\partial x^2}+\frac{\partial^2}{\partial y^2}.
$$
Evidently, every harmonic mapping in a simply connected domain admits the representation $f(z)=A(z)+\overline{B(z)}$. Properties of
harmonic mappings have been investigated extensively (see the monograph of Duren \cite{dur}), especially after the
appearance of the pioneering article of Clunie and Sheil-Small \cite{CS} in 1984. A simple fact is that a harmonic mapping of an
analytic function is harmonic, but an analytic function of a
harmonic mapping is not necessarily harmonic. On the other hand, some elementary facts about
compositions of certain harmonic mappings, which escape the attention of many, were addressed by Reich \cite{rei} in 1987.
For instance the following result was established in \cite{rei}.

\begin{Thm}\label{LP-Th1}
Suppose $f(z)=z+\overline{B(z)}$, where $B(z)$ is analytic and $G(z)=B'(z)$. A necessary and sufficient condition that there locally exists a
non-affine complex harmonic function $g(w)$, such that  $g(f(z))$ is harmonic is that $G(z)$ satisfies
$$(G')^2=\alpha ^2G^4+ 2cG^3+(\overline{\alpha}) ^2G^2
$$
for some complex constant $\alpha$ and some real constant $c$.
\end{Thm}

In addition to Theorem \Ref{LP-Th1}, several important special cases of it were also obtained in \cite{rei} by expressing the analytic
functions $G$ and $B$ in terms of elementary functions.
The present article is motivated by the work of Reich  which provides the local description of all harmonic mappings $f$ such
that $g\circ f$ is harmonic for some non-affine harmonic $g$. Twenty years later,  properties of the composition of
harmonic mappings with harmonic mappings, and the composition of biharmonic mappings with biharmonic mappings
were investigated in \cite{cpw2}. Nevertheless, non-analytic and non-harmonic functions play significant role, eg. biharmonic mappings
in fluid dynamics and elasticity problems (see \cite{hb, khu1, lan}).  However, the investigation of biharmonic mappings in the context of
geometric function theory is a recent one (see \cite{am, amk1, amk2, cpw1, cpw2,   cw}).

We are interested in iterations of the Laplace operator, namely, $p$-harmonic operators
defined inductively by $\Delta^p=\Delta(\Delta^{p-1})$ for $p=2,\ldots$.
A $2p$-times continuously differentiable complex-valued function $f$ in a simply connected domain $D\subseteq\mathbb{C}$ is called
$p$-harmonic in $D$ if $f$ satisfies $p$-harmonic equation $\Delta^pf=0$ in $D$. $f$ is called polyharmonic if it is $p$-harmonic for some $p\in\IN$.
For $p=1$ (resp. $2$), $f$ is  harmonic (resp. biharmonic). Obviously, every $p$-harmonic mapping is $(p+1)$-harmonic.
It is easy to see that, $f$ is $p$-harmonic in a simply connected domain $D$ if and only if (see finite Almansi expansion \cite[Proposition 1.3]{ACL} and \cite[Proposition 1.1]{cpw4})
\be\label{LS1-lem1}
f(z)=\sum_{k=1}^p|z|^{2(k-1)}G_k(z),
\ee
where each $G_k(z)$ is harmonic in $D$.
%, i.e., $\Delta G_k=0$ for $k=1,\ldots,p$.
%Furthermore, the mappings $G_k$ can be expressed as the form
%$$G_k=h_k+\overline{g_k}
%$$
%for each $k\in\{1,\ldots, p\}$, where all $h_k$ and $g_k$ are analytic in $\mathbb{D}$  (see \cite{cpw3, cpw4}).
There is now a long list of articles in the literature on this subject. For recent results on $p$-harmonic mappings, we  refer
to the articles \cite{cpw3, cpw4,  crw1, crw2, crw3, khu2,   lkw, lpw, qw1, qw2}. Another motivation for the study of polyharmonic mappings is from the recent work
of Borichev and  Hedenmalm \cite{BH} on the study of {\it second order elliptic partial differential equations}
$T_{\alpha}(f)=0$
in the unit disk $\mathbb{D}=\{z:\,|z|<1\}$, where $\alpha\in\IR$ and
$$T_\alpha=-\frac{\alpha^2}{4}(1-|z|^2)^{-\alpha-1}
+\frac{\alpha}{2}(1-|z|^2)^{-\alpha-1}(z\frac{\partial}{\partial z}+\overline{z}\frac{\partial}{\partial \overline{z}})
+(1-|z|^2)^{-\alpha}\frac{\partial^2}{\partial z\partial \overline{z}}.
$$
In particular if we take $\alpha=2(p-1)$, then $f$ satisfying $T_{\alpha}(f)=0$ is {\it $p$-harmonic}. Clearly, the choice
$\alpha=0$ gives that $f$ is harmonic.
%Throughout this paper we consider polyharmonic mappings in the unit disk $\mathbb{D}=\{z\in\mathbb{C}~|~|z|<1\}$.
Moreover, the problem of when the composite mappings of $p$-harmonic
mappings with a fixed analytic function are $l$-harmonic was discussed in \cite{lpw}, where $l\in\{1,\ldots,p\}$.

In what follows, the numbers $q,l$ are positive integers unless otherwise stated. $[x]$ denotes the largest integer no
more than $x$, where $x$ is a real number. Recall that $f(z)$ is called an affine mapping if $f(z)=\alpha z+\overline{\beta z}+\gamma$,
where $\alpha$, $\beta$ and $\gamma$ are some constants. Similarly, $f$ is called a harmonic polynomial of degree $n$ if
$f=h+\overline{g}$, where $h$ and $g$ are analytic polynomials such that $n=\max\{{\rm deg\,}h(z),\, {\rm deg\,}g(z)\}$.

In this paper we are mainly concerned with the properties of the composition of polyharmonic mappings.
We completely solve the following problem: \textit{What is the polyharmonic mapping if all its post-compositions by any $q$-harmonic mapping are
$l$-harmonic?} Our main results follow.

\begin{thm}\label{LS1-th1}
Let $f$ be a polyharmonic mapping. Then
\begin{enumerate}
\item[{\rm (a)}] for any analytic $F$, $f\circ F$ is $l$-harmonic if and only if $f$ is harmonic;

\item[{\rm (b)}] for any harmonic $F$ which is not analytic, $f\circ F$ is $l$-harmonic if and only if $f(z)$ is an
affine mapping.
%i.e. $f(z)=\alpha z+\overline{\beta z}+\gamma$, where $\alpha$, $\beta$ and $\gamma$  are constants;

\item[{\rm (c)}] for any $q$-harmonic ($q\geq2$) $F$ which is not $(q-1)$-harmonic, $f\circ F$ is $l$-harmonic
if and only if $f$ is a harmonic polynomial of degree $t$, where $t\leq\min \big \{1,[\frac{l-1}{q-1}]\big \}$.
%$f(z)=\sum_{n=0}^{t}(\alpha_n z^n+\overline{\beta_nz^n})$, where $t=\min\{1,[\frac{l-1}{q-1}]\}$
%and all coefficients of $f$ are constants.
\end{enumerate}
\end{thm}

\begin{cor}\label{LS1-cor1}
Let $f$ be an analytic function. Then
\begin{enumerate}
\item[{\rm (a)}] for any harmonic $F$ which is not analytic, $f\circ F$ is $l$-harmonic if and only if $f(z)$ is linear in $z$.
%$f(z)=\alpha z+\gamma$, where $\alpha$  and $\gamma$ are constants;

\item[{\rm (b)}] for any $q$-harmonic ($q\geq2$) $F$ which is  not $(q-1)$-harmonic, $f\circ F$ is $l$-harmonic
if and only if $f(z)$ is an analytic polynomial of degree $t$, where $t\leq\min\big \{1,[\frac{l-1}{q-1}]\big \}$.
%$f(z)=\sum_{n=0}^{t}\alpha_n z^n$, where $t=\min\{1,[\frac{l-1}{q-1}]\}$ and all coefficients of $f$ are constants.
\end{enumerate}
\end{cor}

We next partly solve the problem of characterizing all polyharmonic mappings if all its pre-compositions by any $q$-harmonic
mapping are $l$-harmonic.

\begin{thm}\label{LS1-th2}
Let $f$ be a harmonic mapping. Then
\begin{enumerate}
\item[{\rm (a)}] for any harmonic $F$, $F\circ f$ is $l$-harmonic if and only if $f$ is analytic or anti-analytic;

\item[{\rm (b)}] for any $q$-harmonic ($q\geq2$) $F$ which is not $(q-1)$-harmonic, $F\circ f$ is $l$-harmonic if and only
if either $f(z)$ or $\overline{f(z)}$ is an analytic polynomial of degree $t$, where $t\leq[\frac{l-1}{q-1}]$.
%$f(z)=\sum_{n=0}^{t}\alpha_n z^n$ or $f(z)=\overline{\sum_{n=0}^{t}\beta_nz^n}$, where
%$t=[\frac{l-1}{q-1}]$  and all  coefficients of $f$ are constants.
\end{enumerate}
\end{thm}

\begin{cor}\label{LS1-cor2}
Let $f$ be an analytic function. Then,
for any $q$-harmonic ($q\geq2$) $F$ but not $(q-1)$-harmonic, $F\circ f$ is $l$-harmonic if and only if
$f(z)$  is an analytic polynomial of degree $t$, where $t\leq[\frac{l-1}{q-1}]$.
%$f(z)=\sum_{n=0}^{t}\alpha_n z^n$, where $t=[\frac{l-1}{q-1}]$.
\end{cor}

\begin{thm}\label{LS1-th3}
Let $f$ be a polyharmonic mapping. Then
\begin{enumerate}
\item[{\rm (a)}] for any harmonic $F$, $F\circ f$ is harmonic (or biharmonic) if and only if $f$ is analytic or anti-analytic;

\item[{\rm (b)}] for any $q$-harmonic ($q\geq2$) $F$ which is not $(q-1)$-harmonic, $F\circ f$ is harmonic if and only if $f(z)$ is identically
constant;

\item[{\rm (c)}] for any $q$-harmonic ($q\geq2$) $F$ which is not $(q-1)$-harmonic, $F\circ f$ is biharmonic if and only if
 either $f(z)$ or $\overline{f(z)}$ is an analytic polynomial of degree $t$, where $t\leq[\frac{1}{q-1}]$.
%$f(z)=\sum_{n=0}^{t}\alpha_n z^n$ or $f(z)=\overline{\sum_{n=0}^{t}\beta_nz^n}$, where $t=[\frac{1}{q-1}]$ and all coefficients of  $f$ are constants.
\end{enumerate}
\end{thm}

The cases where there exist difficulties are formulated as an open problem.

\bcon
Let $f$ be a polyharmonic mapping. Then
\begin{enumerate}
\item[{\rm (a)}] for any harmonic $F$, $F\circ f$ is $l$-harmonic if and only if $f$ is analytic or anti-analytic;

\item[{\rm (b)}] for any $q$-harmonic ($q\geq2$) $F$ which is not $(q-1)$-harmonic, $F\circ f$ is $l$-harmonic if and
only if  either $f(z)$ or $\overline{f(z)}$ is an analytic polynomial of degree $t$, where $t\leq[\frac{l-1}{q-1}]$.
%$f(z)=\sum_{n=0}^{t}\alpha_n z^n$ or $f(z)=\overline{\sum_{n=0}^{t}\beta_nz^n}$, where $t=[\frac{l-1}{q-1}]$ and
%all coefficients of $f$ are constants.
\end{enumerate}
\econ

\section{Some  notations and preliminary results}
%\hspace{0.4cm}
%In this section, we will introduce some notations and prove some propositions which are useful for the proof of theorems.

For simplicity, we introduce the following notations. For $p\in \IN$, we let
$$\mathbb{H}_p=\{f:\, \mbox{$f$ is $p$-harmonic}\},
$$
and $\mathbb{H}_0=\{f:\,\mbox{$f$ is analytic}\}$. Set
$\mathbb{H}_p^*=\mathbb{H}_p\backslash\mathbb{H}_{p-1}$
for $p\geq1$ and $\mathbb{H}_0^*=\mathbb{H}_0$. Obviously,
$$\mathbb{H}_p=\cup_{i=0}^p\mathbb{H}_i^* ~\mbox{ and }~ \mathbb{H}_i^*\cap\mathbb{H}_j^*=\emptyset~(i\neq j).
$$
We observe that if $f$ is $p$-harmonic and is represented by \eqref{LS1-lem1} with $G_p(z)\neq0$ for $p\geq2$, then $f\in\mathbb{H}_p^*$.
%
%\begin{lem} \label{LS1-lem1}
%{\rm(\cite[Proposition 1.1]{cpw4})} A mapping $f$ is $p$-harmonic in $\mathbb{D}$ if and only if $f$ has the following
%representation:  $$f(z)=\sum_{k=1}^p|z|^{2(k-1)}G_k(z),
%$$
%where $G_k(z)$ is harmonic for each $k\in\{1,\ldots,p\}$. Therefore, if $G_p(z)\neq0$ for $p\geq2$, then $f\in\mathbb{H}_p^*$.
%\end{lem}
Finally, we introduce
$$\mathbb{H}_{p,q}^l=\{f:\, \mbox{$f\in\mathbb{H}_p$  such that $F\circ f\in\mathbb{H}_l$ for any $ F\in\mathbb{H}_q^*$}\},
$$
where $p,q,l$ are non-negative integers. Clearly, constant functions belong to $\mathbb{H}_{p,q}^l$.

\begin{prop}\label{LS1-prop1}
We have the following properties.
\begin{enumerate}
\item[{\rm (a)}] $\mathbb{H}_{p,q}^l\subseteq\mathbb{H}_{p+1,q}^l$ and $\mathbb{H}_{p,q}^l\subseteq\mathbb{H}_{p,q}^{l+1}$;

\item[{\rm (b)}] $\mathbb{H}_{p,q}^l\subseteq\mathbb{H}_{p,q-1}^l\subseteq\cdots\subseteq\mathbb{H}_{p,1}^l=\mathbb{H}_{p,0}^l$;

\item[{\rm (c)}] $\mathbb{H}_{p,q}^l=\mathbb{H}_{l,q}^l~(l\leq p)$.
\end{enumerate}
\end{prop}
\begin{proof}
(a) The proofs of the inclusions in (a) are trivial.

(b) For the chain of inclusions in (b), we first show that $\mathbb{H}_{p,q+1}^l\subseteq\mathbb{H}_{p,q}^l$ for any $q\in\mathbb{N}$.
%and $\mathbb{H}_{p,1}^l\supseteq\mathbb{H}_{p,0}^l$.
Assume that $f\in\mathbb{H}_{p,q+1}^l$. Then $f\in\mathbb{H}_p$ and
$F_{q+1}\circ f\in\mathbb{H}_l$ for any $F_{q+1}\in\mathbb{H}_{q+1}^*$. Also, for any $F_q\in\mathbb{H}_q^*$, it is easy to see that
$F_{q+1}+F_q\in\mathbb{H}_{q+1}^*$ and
$$\Delta^l(F_q\circ f)=\Delta^l((F_{q+1}+F_q)\circ f)-\Delta^l(F_{q+1}\circ f)=0.
$$
Thus, $f\in\mathbb{H}_{p,q}^l$ and hence, $\mathbb{H}_{p,q+1}^l\subseteq\mathbb{H}_{p,q}^l$.

For the completion of the proof of the inclusions in (b), we only need to show that $\mathbb{H}_{p,1}^l\supseteq\mathbb{H}_{p,0}^l$.  To do
this, we let $F\in\mathbb{H}_1^*=\mathbb{H}_1\backslash\mathbb{H}_{0}$. Then, $F$ is harmonic with the representation
$F=h+\overline{g}$, where $h$ and $g$ are analytic, but $g$ is not a constant function. Next, we assume $f\in\mathbb{H}_{p,0}^l$. Then
$h\circ f\in\mathbb{H}_l$ and $g\circ f\in\mathbb{H}_l$. Therefore, $F\circ f=h\circ f+\overline{g\circ f}\in\mathbb{H}_l$ and thus,
$f\in\mathbb{H}_{p,1}^l$ which shows that $\mathbb{H}_{p,1}^l\supseteq\mathbb{H}_{p,0}^l$.

(c) We claim that $\mathbb{H}_{p,0}^l=\mathbb{H}_{l,0}^l~(l\leq p)$. First, we prove that $\mathbb{H}_{p,0}^l\subseteq\mathbb{H}_{l,0}^l$,
since $\mathbb{H}_{p,0}^l\supseteq\mathbb{H}_{l,0}^l$ by (a). Next, we assume that $f\in\mathbb{H}_{p,0}^l$. Then $f\in\mathbb{H}_p$ and
$F\circ f\in\mathbb{H}_l$ for any  $F\in\mathbb{H}_0$. Choosing $F(z)=z$, we see that $f=F\circ f\in\mathbb{H}_l\subseteq\mathbb{H}_p$.
Thus, $f\in\mathbb{H}_{l,0}^l$ and thus,  $\mathbb{H}_{p,0}^l\subseteq\mathbb{H}_{l,0}^l$.

Next we show that $\mathbb{H}_{p,q}^l=\mathbb{H}_{l,q}^l~(l\leq p)$. Since $\mathbb{H}_{p,q}^l\subseteq\mathbb{H}_{p,0}^l=\mathbb{H}_{l,0}^l$
by (b), it follows that $\mathbb{H}_{p,q}^l\subseteq\mathbb{H}_{l,q}^l$. Thus, $\mathbb{H}_{p,q}^l=\mathbb{H}_{l,q}^l$ because of the inclusion
$\mathbb{H}_{l,q}^l\subseteq\mathbb{H}_{p,q}^l$ by (a).
\end{proof}

\begin{prop}\label{LS1-prop2}
We have
$\mathbb{H}_{2,0}^2=\{f:\, \mbox{$f$ is  either analytic  or  anti-analytic}\}.$
\end{prop}
\begin{proof} Set $F_m(z)=e^{mz}~(m\neq0)$. Assume that $f\in\mathbb{H}_{2,0}^2$. Then $f\in\mathbb{H}_2$ and $F_m\circ f(z)\in\mathbb{H}_2$,
which shows that $\Delta^2(e^{mf})=0$. By computation, we have that
$$\Delta(e^{mf})=4me^{mf}( f_{z\overline{z}}+mf_zf_{\overline{z}}) ~\mbox{ and }~
\Delta^2(e^{mf}) = 16m^2e^{mf}A_m(z),
$$
where
$$ A_m(z) =2(f^2_{z\overline{z}}+f_zf_{z\overline{z}^2}
+f_{\overline{z}}f_{z^2\overline{z}})+f_{z^2}f_{\overline{z}^2}+m(f^2_zf_{\overline{z}^2}
+f^2_{\overline{z}}f_{z^2}+4f_zf_{\overline{z}}f_{z\overline{z}})+m^2(f_zf_{\overline{z}})^2.
$$

Since $e^{mf}\neq0$ and $m\neq0$, the equation $\Delta^2(e^{mf})=0$ is equivalent to $A_m(z)=0$. In particular,
$A_2(z)-A_1(z)=0$ and $A_3(z)-A_2(z)=0$. That is,
\[\left\{ \begin{array}{ll}
f^2_zf_{\overline{z}^2}
+f^2_{\overline{z}}f_{z^2}+4f_zf_{\overline{z}}f_{z\overline{z}}+3(f_zf_{\overline{z}})^2=0,\\[2mm]
f^2_zf_{\overline{z}^2} +f^2_{\overline{z}}f_{z^2}+4f_zf_{\overline{z}}f_{z\overline{z}}+5(f_zf_{\overline{z}})^2=0.
\end{array}
\right.
\]
Subtracting the first equation from the second gives $(f_zf_{\overline{z}})^2=0$. Therefore, $f$ is either analytic or anti-analytic.
For any analytic function $F$, $F\circ f$ is then analytic or anti-analytic. Obviously, $F\circ f\in\mathbb{H}_2$ which implies
the desired statement of Proposition \ref{LS1-prop2}.
%$$\mathbb{H}_{2,0}^2=\{f:\, \mbox{ either $f$ is  analytic  or  anti-analytic}\}. $$
\end{proof}

%\begin{prop}
%Let $l$ be a positive integer. Then $$\Delta^l(fg)=4^l\sum_{0\leq i,j\leq l}C^l_{i,j}f_{z^i\overline{z}^j}g_{z^{l-i}\overline{z}^{l-j}},$$ where each $C^l_{i,j}$ is a positive integer.
%\end{prop}

%\begin{proof}
%It is easy to verify that $$\Delta(fg)=4(gf_{z\overline{z}}+f_zg_{\overline{z}}+f_{\overline{z}}g_z+fg_{z\overline{z}}).$$
%Thus, the statement is true for $l=1$. Assume the statement is true for $l=k$, then
%\begin{eqnarray*} & &\Delta^{k+1}(fg)= \Delta(\Delta^k(fg))=\Delta(4^k\sum_{0\leq i,j\leq k}C^k_{i,j}f_{z^i\overline{z}^j}g_{z^{k-i}\overline{z}^{k-j}})\\
%&=&
%4^k\sum_{0\leq i,j\leq k}C^k_{i,j}\Delta(f_{z^i\overline{z}^j}g_{z^{k-i}\overline{z}^{k-j}})
%\\
%&=&4^{k+1}\sum_{0\leq i,j\leq k}C^k_{i,j}(f_{z^{i+1}\overline{z}^{j+1}}g_{z^{k-i}\overline{z}^{k-j}}
%+f_{z^{i+1}\overline{z}^{j}}g_{z^{k-i}\overline{z}^{k+1-j}}
%+f_{z^{i}\overline{z}^{j+1}}g_{z^{k+1-i}\overline{z}^{k-j}}+
%f_{z^{i}\overline{z}^{j}}g_{z^{k+1-i}\overline{z}^{k+1-j}})\\
%&=&4^{k+1}\sum_{0\leq i,j\leq k+1}C^{k+1}_{i,j}f_{z^i\overline{z}^j}g_{z^{k+1-i}\overline{z}^{k+1-j}}.\end{eqnarray*}
%It is easy to see that each $C^{k+1}_{i,j}$ is a positive integer. The proof of Proposition 3 is completed.\end{proof}

\section{The proofs of main Theorems}

\subsection{The proof of Theorem \ref{LS1-th1}}
It suffices to prove the necessary parts of the statements (a) to (c), since the sufficiency parts  are obvious. Let
$$f(z)=\sum_{k=1}^p|z|^{2(k-1)}G_k(z)\neq0,
$$
where each $G_k(z)$ is harmonic.

(a)  Assume that $f\circ F\in\mathbb{H}_l$ for any analytic function $F$. Let $F(z)=z^m~(m\in\mathbb{Z}$, $m>l)$ be given. Then
$$f\circ F(z)=\sum_{k=1}^p|z|^{2m(k-1)}G_k(z^m)\in \mathbb{H}_l.
$$
Obviously, each $G_k(z^m)$ is still harmonic. Set
$t=\max\{k:\, G_k(z)\neq0, ~1\leq k\leq p\}.
$
We find  that
$$\sum_{k=1}^p|z|^{2m(k-1)}G_k(z^m)=\sum_{k=1}^t|z|^{2m(k-1)}G_k(z^m)\in \mathbb{H}^*_{m(t-1)+1}.
$$
Thus, by \eqref{LS1-lem1}, we have that $m(t-1)+1\leq l$ which implies that $t=1$ and thus, $f(z)=G_1(z)$ which is harmonic.

(b)  Assume that $f\circ F\in\mathbb{H}_l$ for any $F\in\mathbb{H}_1^*$. Suppose that $F(z)=\overline{z}^m$ $(m\in\mathbb{Z},\, m>l)$ is given.
Then, since for each $k\in\{1,\ldots,p\}$, $G_k(\overline{z^m})=\overline{G_k(z^m)}$ is harmonic, using the similar analysis of the previous case,
$f(z)$ must be harmonic. So, we may let
$$f(z)=\sum_{n=0}^{\infty}\alpha_nz^n+\overline{\sum_{n=0}^{\infty}\beta_nz^n}.
$$
Again, given
$F(z)=z^m+(c\overline{z})^m~(m\in\mathbb{Z}, ~m>l,\, |c|=1),
$
by a straight calculation, we compute that
\begin{eqnarray}
\sum_{n=0}^{\infty}\alpha_n(F(z))^n
%&= &\sum_{n=0}^{\infty}\alpha_n(z^m+(c\overline{z})^m)^n\\
&=&\sum_{n=0}^{\infty}\alpha_n\left (\sum_{k=0}^nC_n^kz^{mk}(c\overline{z})^{m(n-k)}\right) \nonumber\\
&=&\sum_{i,j\geq0}C_{i+j}^j\alpha_{i+j}z^{mi}(c\overline{z})^{mj} \nonumber\\
&=&\sum_{n=0}^{\infty}\alpha_nz^{mn}
+\sum_{n=0}^{\infty}c^{mn}\alpha_n\overline{z}^{mn}
+\sum_{i,j\geq1}^{\infty}C_{i+j}^j\alpha_{i+j}z^{mi}(c\overline{z})^{mj} \nonumber\\
&=&\sum_{n=0}^{\infty}\alpha_nz^{mn}
+\sum_{n=0}^{\infty}c^{mn}\alpha_n\overline{z}^{mn}+\sum_{j=1}^{\infty}C_{2j}^jc^{mj}\alpha_{2j}|z|^{2jm}\nonumber \\
& &+\sum_{i>j\geq1}^{\infty}C_{i+j}^j\alpha_{i+j}(c^{mj}z^{m(i-j)}+c^{mi}\overline{z}^{m(i-j)})|z|^{2jm} \nonumber\\
&=&\sum_{n=0}^{\infty}\alpha_nz^{mn}
+\sum_{n=0}^{\infty}c^{mn}\alpha_n\overline{z}^{mn}
+\sum_{j=1}^{\infty}B_j(z)|z|^{2jm},\label{LS1-eq2}
\end{eqnarray}
where
$$B_j(z)=C_{2j}^jc^{mj}\alpha_{2j}+\sum_{i>j}^{\infty}C_{i+j}^j\alpha_{i+j}\left (c^{mj}z^{m(i-j)}+c^{mi}\overline{z}^{m(i-j)}\right )
$$
for $j\geq1$. Similarly, we have
\be\label{LS1-eq3}
\overline{\sum_{n=0}^{\infty}\beta_n(F(z))^n} %= \overline{\sum_{n=0}^{\infty}\beta_n(z^m+(c\overline{z})^m)^n}
=\sum_{n=0}^{\infty}\overline{\beta}_n\overline{z}^{mn}
+\sum_{n=0}^{\infty}\overline{c}^{mn}\overline{\beta}_nz^{mn}
+\sum_{j=1}^{\infty}C_j(z)|z|^{2jm},
\ee
where
$$C_j(z)=C_{2j}^j\overline{c}^{mj}\overline{\beta}_{2j}+
\sum_{i>j}^{\infty}C_{i+j}^j\overline{\beta}_{i+j}\left (\overline{c}^{mj}\overline{z}^{m(i-j)}+\overline{c}^{mi}z^{m(i-j)}\right )
$$
for $j\geq1$. Adding the last two expressions, namely, \eqref{LS1-eq2} and \eqref{LS1-eq3}, gives
\begin{eqnarray*}
f\circ F(z)
%&= &\sum_{n=0}^{\infty}\alpha_n(z^m+(c\overline{z})^m)^n+\overline{\sum_{n=0}^{\infty}\beta_n(z^m+(c\overline{z})^m)^n}\\
&=&\sum_{n=0}^{\infty}(\alpha_n+\overline{c}^{mn}\overline{\beta}_n)z^{mn}
+\sum_{n=0}^{\infty}(c^{mn}\alpha_n+\overline{\beta}_n)\overline{z}^{mn}
+\sum_{j=1}^{\infty}D_j(z)|z|^{2jm},
\end{eqnarray*}
where  $D_j(z)=B_j(z)+C_j(z)$ for  $j\geq1$, and
\begin{eqnarray*}
D_j(z)&=&C_{2j}^j\left (c^{mj}\alpha_{2j}+\overline{c}^{mj}\overline{\beta}_{2j}\right )\\
& &+\sum_{i>j}^{\infty}C_{i+j}^j\Big( \left  (c^{mj}\alpha_{i+j}+
\overline{c}^{mi}\overline{\beta}_{i+j}\right )z^{m(i-j)}+\left (c^{mi}\alpha_{i+j}+\overline{c}^{mj}\overline{\beta}_{i+j}\right )\overline{z}^{m(i-j)}\Big).
\end{eqnarray*}
Clearly, each $D_j(z)$ is harmonic. Since $f\circ F(z)\in \mathbb{H}_l$ and $m>l$, by \eqref{LS1-lem1},
we have $D_j(z)\equiv0$ for each $j\geq1$. It can be deduced from Parseval's formula that
$$c^{mj}\alpha_{2j}+\overline{c}^{mj}\overline{\beta}_{2j}=c^{mj}\alpha_{i+j}+
\overline{c}^{mi}\overline{\beta}_{i+j}=c^{mi}\alpha_{i+j}+\overline{c}^{mj}\overline{\beta}_{i+j}\equiv0
$$
for each $i>j\geq1$ and every $c$ with $|c|=1$. It follows that
$$\alpha_{2j}=\beta_{2j}=\alpha_{i+j}=\beta_{i+j}=0
$$
for all $i>j\geq1$. Thus, $\alpha_n=\beta_n=0$ for $n\geq2$ which yields that
$$f(z)=\alpha_0+\alpha_1z+\overline{\beta_0+\beta_1z}.
$$

(c) Assume that $f\circ F\in \mathbb{H}_l$ for any $F\in\mathbb{H}_q^*\, (q\geq2)$.

We first claim that $f$ is a harmonic polynomial and then we show that $f$ is either a constant function or an affine mapping.
For this, we begin to consider the representation \eqref{LS1-lem1} with
$$G_k(z)=h_k(z)+\overline{g_k(z)}~(1\leq k\leq p),
$$
where $h_k(z)=\sum_{n=0}^{\infty}\alpha_{k,n}z^n$ and $g_k(z)=\sum_{n=0}^{\infty}\beta_{k,n}z^n$. Choosing
$$F(z)=c|z|^{2(q-1)}~(|c|=1),
$$
we find from \eqref{LS1-lem1} that
\begin{eqnarray*}
f\circ F(z)&=&\sum_{k=1}^p|z|^{4(k-1)(q-1)}\left (\sum_{n=0}^{\infty}c^n\alpha_{k,n}|z|^{2n(q-1)}
+\overline{\sum_{n=0}^{\infty}c^n\beta_{k,n}|z|^{2n(q-1)}}\right )\\
&=& \sum_{k=1}^p\left (\sum_{n=0}^{\infty}(c^n\alpha_{k,n}+\overline{c}^n\overline{\beta}_{k,n})|z|^{2(n+2(k-1))(q-1)}\right )\\
&=& \sum_{k=1}^p\left (\sum_{n=2(k-1)}^{\infty}(c^{n-2(k-1)}\alpha_{k,n-2(k-1)}+\overline{c}^{n-2(k-1)}\overline{\beta}_{k,n-2(k-1)})|z|^{2n(q-1)}\right )\\
%&= &\sum_{n=0}^{\infty}\Big(\sum_{k=1}^p(c^n\alpha_{k,n}
%+\overline{c}^n\overline{\beta_{k,n}})|z|^{2(n+2(k-1))(q-1)}\Big)\\
%&=&
%\sum_{n=0}^{\infty}(c^n\alpha_{1,n}+\overline{c}^n\overline{\beta_{1,n}})|z|^{2n(q-1)}
%+\sum_{n=0}^{\infty}(c^n\alpha_{2,n}+\overline{c}^n\overline{\beta_{2,n}})|z|^{2(n+2)(q-1)}
%+\cdots+\sum_{n=0}^{\infty}(c^n\alpha_{p,n}+\overline{c}^n\overline{\beta_{p,n}})|z|^{2(n+2(p-1))(q-1)}\\
%&=&
%\sum_{n=0}^{\infty}(c^n\alpha_{1,n}+\overline{c}^n\overline{\beta_{1,n}})|z|^{2n(q-1)}
%+\sum_{n=0}^{\infty}(c^n\alpha_{2,n}+\overline{c}^n\overline{\beta_{2,n}})|z|^{2(n+2)(q-1)}
%+\cdots\\& &+\sum_{n=2(p-1)}^{\infty}(c^{n-2(p-1)}\alpha_{p,n-2(p-1)}+\overline{c}^{n-2(p-1)}\overline{\beta_{p,n-2(p-1)}})|z|^{2n(q-1)}\\
&=&\sum_{n=0}^{\infty}A_n|z|^{2n(q-1)},
\end{eqnarray*}
where
\[A_n= \left\{
\begin{array}{ll}
\ds \sum_{k=0}^{[\frac{n}{2}]}(c^{n-2k}\alpha_{k+1,n-2k}+\overline{c}^{n-2k}\overline{\beta}_{k+1,n-2k}) &  \mbox {if $n<2(p-1)$},\\
\ds \sum_{k=0}^{p-1}(c^{n-2k}\alpha_{k+1,n-2k}+\overline{c}^{n-2k}\overline{\beta}_{k+1,n-2k}) &  \mbox {if $n\geq2(p-1)$}.
\end{array}
\right.
\]
Since $f\circ F\in\mathbb{H}_l$, by \eqref{LS1-lem1}, we have  $A_n=0$ for  $n>[\frac{l-1}{q-1}]$.

Next we deduce that all analytic functions $h_k(z)$ and $g_k(z)$ should be polynomials. To do this, we fix $n\in\IN$ such that
$$n\geq\max\Big \{\Big [\frac{l-1}{q-1}\Big], 2(p-1)\Big  \}+1.
$$
Let $c=w^i\, (i=0,1,\ldots,2p-1)$ in $A_n=0$, where $w$ is a primitive $4pn$-th root of unity. Then we have the following equations
$$\sum_{k=0}^{p-1}\Big((w^i)^{n-2k}\alpha_{k+1,n-2k} +(\overline{w}^i)^{n-2k}\overline{\beta}_{k+1,n-2k}\Big)=0, ~~i=0,1,\ldots,2p-1.
$$
It is easy to see that the coefficient determinant of the above equations is a  Vandermonde determinant. As $w$ is a primitive
$4pn$-th root of unity, it is clear that $w^i\neq w^j$ ($i\neq j$) and $w^i\neq \overline{w}^j$  for $i,j\in\{0,1,\ldots,n(2p-1)\}$.
Thus, this  Vandermonde determinant does not vanish. It follows that
$$\alpha_{k+1,n-2k}=\overline{\beta}_{k+1,n-2k}=0 ~ \mbox{for $k=0,1,\ldots,p-1$}.
$$
Therefore, all $h_k(z)$ and $g_k(z)~(1\leq k\leq p)$ are polynomials. If $p=1$, then it is clear that $f$ is a harmonic polynomial.
For $p\geq2$, we will prove that $f$ is also a harmonic polynomial.

Without loss of generality, we may let $h_k(z)=\sum_{n=0}^d\alpha_{k,n}z^n$ and $g_k(z)=\sum_{n=0}^d\beta_{k,n}z^n$ $(d\geq2, 1\leq k\leq p)$.
Again, we choose
$$F(z)=z^{(m+1)(q-1)}\overline{z}^{q-1}=|z|^{2(q-1)}z^{m(q-1)}\in\mathbb{H}_q^*~~(m\in\mathbb{Z},\, m>l+d).
$$
By computation, we have
\begin{eqnarray*}
& &f\circ F(z)\\&=& \sum_{k=1}^p\left |\,|z|^{2(q-1)}z^{m(q-1)}\right |^{2(k-1)}\left (\sum_{n=0}^d\alpha_{k,n}(|z|^{2(q-1)}z^{m(q-1)})^n
+\overline{\sum_{n=0}^d\beta_{k,n}(|z|^{2(q-1)}z^{m(q-1)})^n}\right )\\
&=&\sum_{k=1}^p\sum_{n=0}^{d}\Big(\alpha_{k,n}z^{nm(q-1)}
+\overline{\beta}_{k,n}\overline{z}^{nm(q-1)}\Big)|z|^{2(t_k+n)(q-1)}\\
&=&\sum_{k=1}^p\sum_{n=t_k}^{t_k+d}\Big(\alpha_{k,n-t_k}z^{(n-t_k)m(q-1)}
+\overline{\beta}_{k,n-t_k}\overline{z}^{(n-t_k)m(q-1)}\Big)|z|^{2n(q-1)}\\
%&=&\sum_{n=0}^{d}(\alpha_{1,n}z^{n(m-1)(q-1)}
%+\overline{\beta_{1,n}}\overline{z}^{n(m-1)(q-1)})|z|^{2n(q-1)}+\cdots\\
%& &+\sum_{n=(m+1)(p-1)}^{(m+1)(p-1)+d}\big(\alpha_{p,n-(m+1)(p-1)}z^{(n-(m+1)(p-1))(m-1)(q-1)}\\
%& &~~~~~~~~~~~~~~~~~~~~+\overline{\beta_{p,n-(m+1)(p-1)}}\overline{z}^{(n-(m+1)(p-1))(m-1)(q-1)}\big)|z|^{2n(q-1)}\\
&=&\sum_{n=0}^{(m+2)(p-1)+d}E_n(z)|z|^{2n(q-1)},
\end{eqnarray*}
where $t_k=(m+2)(k-1)$ for $k\in\{1,\ldots,p\}$ and
\[E_n(z)= \left\{
\begin{array}{ll}
\alpha_{k,t}z^{tm(q-1)}+\overline{\beta_{k,t}}\overline{z}^{tm(q-1)} & \mbox{ for $n=(m+2)(k-1)+t$ $(t\in\{0,1,\ldots,d\})$},\\
0 & \mbox{ otherwise},
\end{array}
\right.
\]
where $k\in\{1,\ldots,p\}$. Obviously, all $E_n(z)$'s are harmonic polynomials. Since $f\circ F(z)\in\mathbb{H}_l$ and $(m+2)(k-1)>l$
for $k\in\{2,\ldots,p\}$, by \eqref{LS1-lem1}, we have $E_n(z)\equiv0$ for $n\geq m+2$. Again, by Parseval's formula, it follows that
$$\alpha_{k,t}=\beta_{k,t}=0
$$
for $k\in\{2,\ldots,p\}$ and $t\in\{0,\ldots,d\}$. Therefore, $f(z)$ is a harmonic polynomial and thus, for simplicity, we may write it as
$$f(z)=\sum_{n=0}^{d}(\alpha_n z^n+\overline{\beta_nz^n})~(d>1).
$$

Now, we show that $f(z)$ is either a constant function or an affine mapping. To do this, we choose
$F(z)=c|z|^{2(q-1)}$, where $|c|=1$  and see that
\begin{eqnarray*}
f\circ F(z)
%&=&\sum_{n=0}^{d}c^n\alpha_n|z|^{2n(q-1)}+\overline{\sum_{n=0}^{d}c^n\beta_n|z|^{2n(q-1)}}\\
&=&\sum_{n=0}^{d}(c^n\alpha_n+\overline{c}^n\overline{\beta}_n)|z|^{2n(q-1)}\in\mathbb{H}_l.
\end{eqnarray*}
By the representation \eqref{LS1-lem1}, we conclude that
$$c^n\alpha_n+\overline{c}^n\overline{\beta}_n=0
$$
for any $n>\frac{l-1}{q-1}\geq[\frac{l-1}{q-1}] = t$ and $c$ with $|c|=1$. This gives $\alpha_n=\beta_n=0$ for $n>t$.
Next we divide the proof into three cases.

(i) If $t=0$,  then $l<q$ and thus, $f(z)=\alpha_0+\overline{\beta_0}$ which is a constant. Obviously, $f\circ F(z)$ reduces to a constant and hence,
belongs to $\in\mathbb{H}_l$ for any $F(z)\in\mathbb{H}_q^*$.

(ii) If $t=1$, then $q\leq l\leq 2q-2$ and $f(z)=\alpha_0+\alpha_1z+\overline{\beta_0+\beta_1z}$. Obviously,
$f\circ F(z)\in\mathbb{H}_q\subset\mathbb{H}_l$ for any $F(z)\in\mathbb{H}_q^*$.

(iii) If $t\geq2$, then $l>2q-2$. We claim that $\alpha_n=\beta_n=0$ for $2\leq n\leq t$. Again, let
$$F(z)=c|z|^{2(q-1)}(z^{2m}+\overline{z}^{2m} )~ (m\in\mathbb{Z},~m>l,~|c|=1).
$$
Then, again by \eqref{LS1-lem1}, we have
\begin{eqnarray*}
f\circ F(z)
%&=& \sum_{n=0}^{t}\alpha_n\big(c|z|^{2(q-1)}(z^{2m}+\overline{z}^{2m})\big)^n
%+\overline{\sum_{n=0}^{t}\beta_n\big(c|z|^{2(q-1)}(z^{2m}+\overline{z}^{2m})\big)^n}\\
%&=&
%\alpha_0+\alpha_1\big(c|z|^{2(q-1)}(z^{2m}+\overline{z}^{2m})\big)+\sum_{n=2}^{t}c^n\alpha_n|z|^{2n(q-1)}(z^{2m}+\overline{z}^{2m})^n\\
%& &+
%\overline{\beta_0}+\overline{\beta_1}\big(\overline{c}|z|^{2(q-1)}(z^{2m}+\overline{z}^{2m})\big)
%+\sum_{n=2}^{t}\overline{c}^n\overline{\beta_n}|z|^{2n(q-1)}(z^{2m}+\overline{z}^{2m})^n\\
&=&\sum_{n=0}^1(c^n\alpha_n+\overline{c}^n\overline{\beta}_n)|z|^{2n(q-1)}(z^{2m}+\overline{z}^{2m})^n\\
&& \qquad +\sum_{n=2}^{t}(c^n\alpha_n+\overline{c}^n\overline{\beta}_n)|z|^{2n(q-1)}(z^{2m}+\overline{z}^{2m})^n\\
&=& H(z)+\sum_{n=2}^{t}(c^n\alpha_n+\overline{c}^n\overline{\beta}_n)H_n(z),
\end{eqnarray*}
where
$$H(z)=\sum_{n=0}^1(c^n\alpha_n+\overline{c}^n\overline{\beta}_n)\,|z|^{2n(q-1)}(z^{2m}+\overline{z}^{2m})^n
%$$
~\mbox{ and }~ H_n(z)=|z|^{2n(q-1)}(z^{2m}+\overline{z}^{2m})^n.
$$
Clearly, $H(z)\in\mathbb{H}_{q}$. However,  if $n\geq2$ is an even number, then one writes
\begin{eqnarray*}
H_n(z)&=&|z|^{2n(q-1)}\sum_{k=0}^nC_n^kz^{2m(n-k)}\overline{z}^{2mk}\\
&=&|z|^{2n(q-1)}\Big(C_n^{\frac{n}{2}}|z|^{2nm}+\sum_{k=0}^{\frac{n}{2}-1}(C_n^kz^{2m(n-2k)}|z|^{4mk}
+C_n^{\frac{n}{2}+k+1}\overline{z}^{4m(k+1)}|z|^{2m(n-2-2k)})\Big),
\end{eqnarray*}
and thus,  it follows from \eqref{LS1-lem1} that $H_n(z)\in\mathbb{H}_{n(m+q-1)+1}^*$ and $H_n(z)$ can be expressed as
$$H_n(z)=C_n^{\frac{n}{2}}|z|^{2n(m+q-1)}+\widetilde{H_n}(z),
$$
%where $\widetilde{H_n}(z)\in \mathbb{H}_{n(m+q-1)}^*$.
where $\widetilde{H_n}(z)\in \mathbb{H}_{n(m+q-1)}$.
 If $n\geq2$ is an odd number, then it is easy to deduce that
$H_n(z)\in\mathbb{H}_{n(m+q-1)-m+1}^*$ and
$$H_n(z)=C_n^{\frac{n-1}{2}}(z^{2m}+\overline{z}^{2m})|z|^{2(n(m+q-1)-m)}+\widehat{H_n}(z),
$$
%where $\widehat{H_n}(z)\in \mathbb{H}_{n(m+q-1)-m}^*$.
where $\widehat{H_n}(z)\in \mathbb{H}_{n(m+q-1)-m}$.
Note that both $n(m+q-1)+1$ and $n(m+q-1)-m+1$ are strictly monotonically increasing as $n$ from 2 to $t$.  Moreover, there are
no integers $n_1,~n_2\in[2,t]$ such that
$n_1(m+q-1)+1=n_2(m+q-1)-m+1$. Therefore, if $t$ is an even number, then $f\circ F$ has the following form
$$f\circ F(z)=C_t^{\frac{t}{2}}(c^t\alpha_t+\overline{c}^t\overline{\beta}_t)|z|^{2t(m+q-1)}+L_t(z),
$$
where $L_t(z)\in\mathbb{H}_{t(m+q+1)}$. Since $f\circ F\in\mathbb{H}_l$ and $t(m+q-1)+1>l+1$, by \eqref{LS1-lem1},
we get $$c^t\alpha_t+\overline{c}^t\overline{\beta}_t=0,
$$
for any $c$ with modulus one. Thus, $\alpha_t=\beta_t=0$. If $t>2$, then $t-1$ is a odd number. Therefore, $f\circ F$ can be written as
$$f\circ F=C_{t-1}^{\frac{t-2}{2}}(c^{t-1}\alpha_{t-1}+\overline{c}^{t-1}\overline{\beta}_{t-1})(z^{2m}+\overline{z}^{2m})|z|^{2((t-1)(m+q-1)-m)}+L_{t-1}(z),
$$
where $L_{t-1}(z)\in\mathbb{H}_{(t-1)(m+q-1)-m}$. Since $f\circ F\in\mathbb{H}_l$ and  $(t-1)(m+q+1)-m>l+1$,
by \eqref{LS1-lem1}, we have that
$$c^{t-1}\alpha_{t-1}+\overline{c}^{t-1}\overline{\beta}_{t-1}=0,
$$
for any $c$ with modulus one. Thus, $\alpha_{t-1}=\beta_{t-1}=0$. If $t=3$, the proof is finished. If $t>3$,
we can similarly obtain that
$$\alpha_{t-2}=\beta_{t-2}=\cdots =\alpha_2=\beta_2=0,
$$
since $n(m+q-1)+1$ and $n(m+q-1)-m+1$ are greater than $l+1$ for any $n\in\{2,\ldots,t\}$.

If $t$ is an odd number, then, by a similar analysis in the even case, we get that
$$\alpha_t=\beta_t=\cdots =\alpha_2=\beta_2=0.
$$
In other words, if $f\circ F\in\mathbb{H}_l$ for any $F\in\mathbb{H}_q^*~(q\geq2)$, then
$f(z)=\sum_{n=0}^{t}\alpha_n z^n+\overline{\sum_{n=0}^{t}\beta_n z^n}$, where $t\leq\min\{1,[\frac{l-1}{q-1}]\}$.
The proof of Theorem \ref{LS1-th1} is finished. \hfill $\Box$

%\textbf{Remark} The Theorem 3 shows that for any positive integers $l$ and $q(\geq2)$, $$\mathbb{H}_{0,q}^l=\{f|~f(z)=\sum_{n=0}^t\alpha_nz^n,~t=[\frac{l-1}{q-1}]\}.$$

\subsection{The proof of Theorem \ref{LS1-th2}}
  The statement (a) is equivalent to proving
$$\mathbb{H}_{1,0}^l\cap\mathbb{H}_{1,1}^l=\{f:\, \mbox{$f$ is  either analytic  or  anti-analytic}\}.
$$
By Proposition \ref{LS1-prop1}, we have $\mathbb{H}_{1,1}^l=\mathbb{H}_{1,0}^l$ and thus, it suffices to prove
$$\mathbb{H}_{1,0}^l=\{f:\, \mbox{$f$ is  either analytic or anti-analytic}\}.
$$
Let $f\in\mathbb{H}_{1,0}^l$. Then, $f(z)$ is harmonic and  thus,  has the form $f(z)=h(z)+\overline{g(z)}$,
where $h(z)$ and $g(z)$ are analytic. Set
$$F(z)=(e^z)^m~(m\in\mathbb{Z},m>l).
$$
Then we get that
$$F\circ f(z)=(e^{h(z)}\overline{e^{g(z)}})^m =H^m(z)\overline{G^m(z)},
$$
where $H(z)=e^{h(z)}$ and $G(z)=e^{g(z)}$ are analytic. Since $F\circ f \in\mathbb{H}_l$, we see that
$$\Delta^l(F\circ f(z))=\Delta^l(H^m(z)\overline{G^m(z)})
=4^l\frac{\partial^lH^{m}(z)}{\partial z^l}\frac{\partial^l\overline{G^{m}(z)}}{\partial \overline{z}^l}=0,
$$
which yields
$$\mbox{ either }~\frac{\partial^lH^{m}(z)}{\partial z^l}=0~\mbox{ or } ~\frac{\partial^lG^{m}(z)}{\partial z^l}=0.
$$
Thus, either $H^m(z)$ or $G^m(z)$ is a polynomial with degree not more than $l-1$. Note that $m>l$. Consequently, either
$H(z)$ or $G(z)$ is a constant which in turn implies that either $h(z)$ or $g(z)$ is a constant. Therefore, $f(z)$ is analytic or anti-analytic.
Obviously, for any harmonic $F$ and any analytic or anti-analytic $f$, $F\circ f$ is harmonic which is also $l$-harmonic.

(b) The statement (b) is equivalent to showing that
$$\mathbb{H}_{1,q}^l=\left \{f:\, f=\sum_{n=0}^{t}\alpha_n z^n ~\mbox{ or }~ f(z)=\overline{\sum_{n=0}^{t}\beta_nz^n}\right \},
$$
where $q\geq 2$ and $t\leq[\frac{l-1}{q-1}]$. By Proposition \ref{LS1-prop1}, we have $\mathbb{H}_{1,q}^l\subseteq\mathbb{H}_{1,0}^l$.
Thus, if $f\in\mathbb{H}_{1,q}^l$,  then either $f$ is analytic or anti-analytic.

Since $F\circ\overline{f}=\overline{\overline{F}\circ f}$, we only need to consider the case that $f$ is analytic.
Let $F(z)=|z|^{2(q-1)}(\in\mathbb{H}_q^*)$. Then we have that
$$H(z)= F\circ f(z)=f^{q-1}(z)\overline{f}^{q-1}(z).
$$
If $H\in\mathbb{H}_l$, then we have
$$\Delta^lH(z)
%=\Delta^l(f^{q-1}(z)\overline{f}^{q-1}(z))
=4^l\frac{\partial^lf^{q-1}(z)}{\partial z^l}\frac{\partial^l\overline{f}^{q-1}(z)}{\partial \overline{z}^l}=0,
$$
which yields $\frac{\partial^lf^{q-1}(z)}{\partial z^l}=0$. Therefore, $f^{q-1}(z)=\sum_{n=0}^{l-1}\beta_nz^n$ which implies
that $f$ must be a polynomial of degree not more than $[\frac{l-1}{q-1}]$ and thus,
we write $f(z)=\sum_{n=0}^{t}\alpha_n z^n~(t\leq[\frac{l-1}{q-1}])$. For any $F\in\mathbb{H}_q^*~(q\geq2)$,
$F$ has the representation $$F(z)=\sum_{k=1}^q|z|^{2(k-1)}G_k(z),
$$
where each $G_k(z)$ is harmonic and $G_q(z)\neq0$. By computation, we may then write
\begin{eqnarray*}
F\circ f(z)&= &\sum_{k=1}^q\Big(\sum_{n=0}^{t}\alpha_n z^n\Big)^{k-1}\overline{\Big(\sum_{n=0}^{t}\alpha_n z^n\Big)^{k-1}}G_k(f(z))\\
&=&\sum_{k=1}^q\Big(\sum_{n=0}^{t(k-1)}c_n z^n\Big)\Big(\sum_{n=0}^{t(k-1)}\overline{c_n}~\overline{z}^n\Big)G_k(f(z))
~\mbox{ for some $c_n$'s}\\
&=&\sum_{k=1}^q\Big(\sum_{0\leq i,j\leq t(k-1)}C_{ij}z^i\overline{z}^jG_k(f(z))\Big),
\end{eqnarray*}
where each $C_{i,j}$ is a complex number. It is easy to deduce that $z^i\overline{z}^jG_k(f(z))$ is $\max\{i+1,j+1\}$-harmonic.
Since $i,j\leq t(q-1)$, we see that $F\circ f\in\mathbb{H}_{t(q-1)+1}$. As $t(q-1)+1\leq l$, it is obvious that
$F\circ f\in\mathbb{H}_{l}$. The proof is complete. \hfill $\Box$

\subsection{The proof of Theorem \ref{LS1-th3}}
The sufficiency parts of the statements in (a)-(c) are obvious and therefore, we need to prove only the necessary part of the theorem.

(a)  The statement (a) is equivalent to proving
$$\mathbb{H}_{p,0}^1\cap\mathbb{H}_{p,1}^1=\mathbb{H}_{p,0}^2\cap\mathbb{H}_{p,1}^2=\{f:\, \mbox{$f$ is  either analytic or anti-analytic}\},
$$
where $p\geq2$. By Proposition \ref{LS1-prop1} and Theorem \ref{LS1-th2}, we acquire
$$\mathbb{H}_{p,1}^1=\mathbb{H}_{p,0}^1=\mathbb{H}_{1,0}^1=\{f:\, \mbox{$f$ is  either analytic or anti-analytic}\}.
$$
By Propositions \ref{LS1-prop1} and \ref{LS1-prop2}, we get
$$\mathbb{H}_{p,1}^2=\mathbb{H}_{p,0}^2=\mathbb{H}_{2,0}^2=\{f:\, \mbox{$f$ is  either analytic or anti-analytic}\}.
$$

(b) By Proposition \ref{LS1-prop1}, we have $\mathbb{H}_{p,q}^1=\mathbb{H}_{1,q}^1$. Therefore, the statement is an
immediate consequence of Theorem \ref{LS1-th2}.

(c) By Proposition \ref{LS1-prop1}, we obtain that $\mathbb{H}_{p,q}^2=\mathbb{H}_{2,q}^2\subseteq\mathbb{H}_{2,0}^2$.
By Proposition \ref{LS1-prop2}, it follows that $\mathbb{H}_{p,q}^2\subseteq\mathbb{H}_{1,q}^2$ and the desired conclusion
follows.  \hfill $\Box$

\subsection*{Acknowledgments}    The research was supported by the NSFs of China~(No.11371363), the construct program of the key discipline in Hunan province and the Hunan Provincial Natural Science Foundation of China (No.2015JJ6011).
%The authors would like
%to thank   the referees for their valuable comments for improving this
%paper.

\end{document}